\documentclass{amsart}

\usepackage[T1]{fontenc}
\usepackage[utf8]{inputenc}

\usepackage[noadjust]{cite}
\usepackage{palatino}
\usepackage{hyperref}
\usepackage{amssymb,amsthm,units,mathabx}
\usepackage{color}
\usepackage[usenames,dvipsnames,svgnames,table]{xcolor}
\usepackage{mathtools,xparse,xspace}
\usepackage{rotating}
\usepackage{bm} 

\hypersetup{allbordercolors=red,linkcolor=red,urlcolor=blue,urlbordercolor=blue,colorlinks=false,pdfborder={0 0 1}} 

\usepackage{epsfig}  		
\usepackage{epic,eepic}       
\usepackage{tikz,pgf}
\usetikzlibrary{arrows,positioning}
\usepackage[abs]{overpic}

\newtheorem{thm}{Theorem}[section]

\newtheorem{lem}[thm]{Lemma}
\newtheorem{cor}[thm]{Corollary}

\newtheorem{question}[thm]{Question}

\theoremstyle{definition}
\newtheorem{definition}[thm]{Definition}
\newtheorem{example}[thm]{Example}

\theoremstyle{remark}

\newtheorem{remark}[thm]{Remark}

\numberwithin{equation}{section}

\newcommand{\R}{\mathbb{R}}  
\newcommand{\Z}{\mathbb{Z}}  
\newcommand{\Q}{\mathbb{Q}}  
\newcommand{\F}{\mathbb{F}}  
\newcommand{\bd}{\partial}  

\newcommand{\chat}{\widehat{c}} 
\newcommand{\cplus}{c^+} 
\newcommand{\cplusred}{c^+_{\rm{red}}} 
\newcommand{\HFhat}{\widehat{{HF}}} 
\newcommand{\CFhat}{\widehat{{CF}}} 
\newcommand{\HFplus}{{{HF}}^+} 
\newcommand{\CFplus}{{{CF}}^+} 
\newcommand{\HFminus}{{{HF}}^-} 
\newcommand{\CFminus}{{{CF}}^-} 
\newcommand{\HFpm}{{{HF}}^{\pm}} 
\newcommand{\HFmp}{{{HF}}^{\mp}} 
\newcommand{\HFred}{{{HF}}_{\mathrm{red}}} 
\newcommand{\HFinfinity}{{{HF}}^{\infty}} 
\newcommand{\CFinfinity}{{{CF}}^{\infty}} 
\newcommand{\HF}{{{HF}}} 
\newcommand{\CF}{{{CF}}} 
\newcommand{\Spinc}{\mathrm{Spin}^{\it c}}
\newcommand{\VO}{{\rm V}} 
\newcommand{\DVO}{\widetilde{\rm V}} 

\newcommand{\spins}{\mathfrak{s}}

\DeclareMathOperator{\im}{im}
\DeclareMathOperator{\rank}{rank}
\DeclareMathOperator{\coker}{coker}

\begin{document}

\begin{abstract}
We investigate the line between tight and overtwisted for surgeries on fibred transverse knots in contact 3-manifolds.  When the contact structure $\xi_K$ is supported by the fibred knot $K \subset M$, we obtain a characterisation of when negative surgeries result in a contact structure with non-vanishing Heegaard Floer contact class.  To do this, we leverage information about the contact structure $\xi_{\overline{K}}$ supported by the mirror knot $\overline{K} \subset -M$.  We derive several corollaries about the existence of tight contact structures, L-space knots outside $S^3$, non-planar contact structures, and non-planar Legendrian knots.
\end{abstract}

\title{Tight Contact Structures via Admissible Transverse Surgery}
\author{James Conway}
\address{University of California, Berkeley}
\email{conway@berkeley.edu}
\maketitle

\section{Introduction}

Contact geometers have long been fascinated with the subtle art of distinguishing tight contact structures from overtwisted one: the border between these two categories is poorly understood.  In this paper, we investigate one manifestation of this problem: when does transverse surgery on transverse knots in overtwisted contact structures result in a tight contact structure?

We focus on fibred knots $K \subset M$, where the fibre surface is a (rational) Seifert surface for $K$; to such a knot is naturally associated a contact structure $\xi_K$, turning $K$ into a transverse knot.  Both null-homologous fibred knots and rationally fibred knots ({\it ie.\ }rationally null-homologous and fibred) support a unique contact structure up to isotopy (see \cite{Giroux:OBD, BEVHM}).

What draws our attention to such knots is a result of Etnyre and Vela-Vick (\cite{EVV:torsion}), namely, that $\xi_K$ restricted to the complement of $K$ is always tight, even when $(M, \xi_K)$ is overtwisted.  Following the guiding principle that ``removing twisting'' of the contact planes increases the chances of being tight, we will investigate the operation of \textit{admissible transverse surgery}: this operation defines a contact structure on $M_r(K)$ that has ``less twisting'' of the contact planes than the original contact structure.

We know that this operation has a hope of succeeding, because when we remove a ``sufficient amount'' of twisting, it is known that this operation will result in a tight contact structure (\textit{cf.\@} Corollary~\ref{FDTC > 1 tight}). We will look at what we can say when an arbitrary small amount of twisting is removed.  In this setting, the more negative the surgery coefficient, the less twisting is removed.

To make the above more precise, we say a few words about how the topological operation of surgery interacts with the contact geometry. Topological $r$-surgery on a null-homologous fibred knot $K$ induces a rational fibred knot $K_r$ in $M_r(K)$, for $r \neq 0$.  The contact structure $\xi_{K_r}$ on $M_r(K)$ induced by the fibred knot $K_r$ was determined by Baker, Etnyre, and van Horn-Morris (in \cite{BEVHM}) (for $r < 0$) and by the author (\cite{Conway}) (for $r > 0$) to be the contact structure obtained by admissible (resp.\@ inadmissible) transverse $r$-surgery on $K$ in $(M,\xi_K)$.  In the cases we will consider, admissible (resp.\@ inadmissible) transverse surgery on $K$ corresponds to negative (resp.\@ positive) contact surgery on a Legendrian approximation of $K$ (this is not true in general; see \cite{BE:transverse}).

The main result of this paper is a characterisation of the Heegaard Floer contact invariant of the contact structure $\xi_{K_r}$, for $r < 0$. In particular, we will be concerned with $\cplus(\xi_r(K)) \in \HFplus(-M_r(K))$ and its image $\cplusred(\xi_{K_r})$ in $\HFplus(-M_r(K)) \to \HFred(-M_r(K))$, defined by Ozsváth and Szabó in \cite{OS:contact}.  Among other properties, the non-vanishing of either of these classes implies that the contact structure $\xi_{K_r}$ is tight.

To state the result, we also need to consider the ``mirror'' of $K$ as a fibred knot in $-M$ (that is, $M$ with reversed orientation).  To emphasise the ambient manifold, we will denote this knot as $\overline{K} \subset -M$, and the supported contact structure on $-M$ as $\xi_{\overline{K}}$.

\begin{thm}\label{maintheorem}
If $(M_r(K), \xi_{K_r})$ is the result of admissible transverse $r$-surgery on $K \subset (M, \xi_K)$, then:
\begin{enumerate}
\item $c^+(\xi_{K_r}) \neq 0$ for all $r < 0$ if and only if $c^+_{\rm{red}}(\xi_{\overline{K}}) = 0$.
\item $c^+_{\rm{red}}(\xi_{K_r}) \neq 0$ for all $r < 0$ if and only if $c^+(\xi_{\overline{K}}) = 0$.
\end{enumerate}
\end{thm}

This theorem will proved in slightly more generality (taking surgery rationally fibred knots into account) as Theorem~\ref{rationalmaintheorem}, and in Corollary~\ref{Rplus(red)}, we look at the particular values of $r$ at which the contact invariants change from zero to non-zero.

Ozsváth and Szabó took the first steps toward this result: in \cite{OS:contact}, they showed that when $\HFred(M) = 0$, then $\cplus(\xi_{K_{-1}}) \neq 0$ for any fibred knot $K \subset M$.  Since $\HFred(M) = 0$ implies that $\cplusred(\xi_{\overline{K}}) = 0$, our result extends this to $\xi_{K_r}$, for all $r < 0$.  Similar results appear by Hedden and Plamenevskaya in \cite{HP}, in the context of L-space surgeries.  In \cite{HM}, Hedden and Mark investigated $\cplus(\xi_{K_r})$ for $r = -1/n$, in cases where Theorem~\ref{maintheorem} gives no information. It is always true that $\cplus(\xi_{K_{-1/n}})\neq 0$ for sufficiently large positive integers $n$ (see Corollary~\ref{FDTC > 1 tight}), and they determined upper bounds on the minimum such $n$ (\textit{cf.\@} Corollary~\ref{Rplus(red)}).

\begin{example}
For an example of how Theorem~\ref{maintheorem} can obviate the need for detailed calculation of Heegaard Floer groups, we show that the open book $(\Sigma, \phi)$ in Figure~\ref{genus 2 example} supports a tight contact structure with $\cplusred \neq 0$, for any integers $k_1, k_2 \geq 2$.  Let $K$ be the binding of this open book, and consider the open book $(\Sigma,\phi^{-1}\circ \tau_{\bd})$ for inadmissible transverse $(-1)$-surgery on $\overline{K}$, where $\tau_{\bd}$ is a boundary parallel Dehn twist.  This latter open book is a stabilisation of one considered in \cite[Section~4]{BE:transverse}, which was shown to support a contact structure with $\cplus = 0$.  By Theorem~\ref{maintheorem}, we conclude that $(\Sigma,\phi)$ supports a tight contact structure with $\cplusred \neq 0$.

\begin{figure}[htbp]
\begin{center}
\begin{overpic}[scale=2,tics=20]{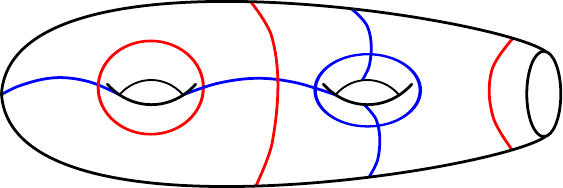}
\put(30,66){\large $-$}
\put(80,90){\large $+$}
\put(135,66){\large $-$}
\put(160,90){\large $+$}
\put(215,90){\large $-k_1$}
\put(222,17){\large $-k_2$}
\put(250,60){\large $-$}
\put(288,60){\large $+$}
\end{overpic}
\caption{A genus-2 open book supporting a tight contact structure with $\cplusred \neq 0$, for any integers $k_1, k_2 \geq 2$.}
\label{genus 2 example}
\end{center}
\end{figure}
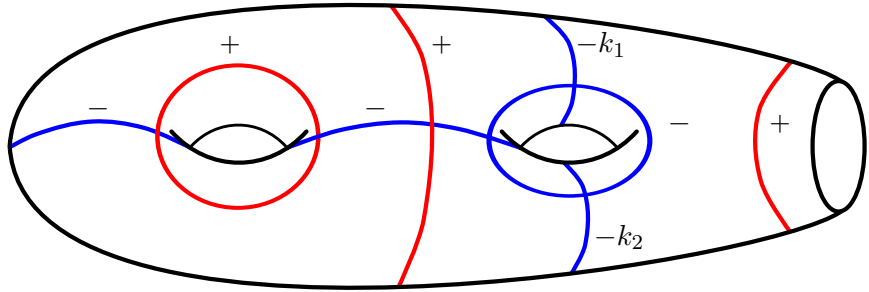
\end{example}

We present several consequences of Theorem~\ref{maintheorem}.  First, we can recover a familiar result of Honda, Kazez, and Matić.  The fractional Dehn twist coefficient is a measure of how much a mapping class ``twists'' arcs with endpoints on the boundary components of an open book.

\begin{cor}[\cite{HKM:right2}]\label{FDTC > 1 tight}
Let $(\Sigma, \phi)$ be an open book decomposition, where $\Sigma$ has one boundary component.  If the fractional Dehn twist coefficient (FDTC) of $\phi$ around $\bd \Sigma$ is greater than $1$, then the supported contact structure is tight with $\cplusred \neq 0$.
\end{cor}

Their proof involved the construction first of a taut foliation which was then perturbed to a tight contact structure. Our proof does not require passing via taut foliations, nor does it prove their existence (in that sense, we do not recover the full extent of their result).

We call $K \subset M$ an L-space knot if $M_r(K)$ is an L-space for some $r > 0$, \textit{ie.\@} $\big|H_1(M_r(K))\big| = \rank \HFhat(M_r(K))$.  In $S^3$, such knots are known to be fibred (\cite{Ni}) and support the tight contact structure (\cite{Hedden:positivity}).  Outside of $S^3$, not all L-space knots are fibered.  However, when they \textit{are} fibered, we generalise Hedden's result from \cite{Hedden:positivity}.

\begin{cor}\label{L-space knot}
If $K \subset M$ is a fibered L-space knot, then $\cplus(\xi_K) \neq 0$.
\end{cor}

The support genus of a contact structure is defined in \cite{EO:invariants} to be the minimal genus of an open book supporting the contact structure.  Ozsváth, Stipsicz, and Szabó showed in \cite{OSS:planar} that if $\cplusred(\xi) \neq 0$, then $\xi$ is non-planar, \textit{ie.\@} cannot be supported by an open book with planar pages.

\begin{cor}\label{non-planar contact structures}
If $K \subset M$ is fibred, and $\cplus(\xi_{\overline{K}}) = 0$, then $(M_r(K), \xi_{K_r})$ is non-planar for all $r < 0$.
\end{cor}

Similarly, consider the support genus of a Legendrian knot $L \subset (M, \xi)$, \textit{ie.\@} the minimal genus open book supporting $\xi$ such that $L$ sits on a page, defined in \cite{Onaran}.  This next corollary generalises a result of Li and Wang (\cite{LW}).

\begin{cor}\label{non-planar Legendrian knots}
If $K \subset M$ is fibred, and $\cplus(\xi_{\overline{K}}) = 0$, and $L$ is a Legendrian approximation of $K$, then neither $L$ nor $-L$ sit on the page of a planar open book supporting $\xi_K$, where $-L$ is $L$ with the orientation reversed.
\end{cor}

In particular, no Legendrian approximation of a torus knot (or more generally, any non-trivial fibered strongly quasi-positive knot, by \cite{Hedden:positivity}) with maximal Thurston--Bennequin number sits on the page of a planar open book for $(S^3, \xi_{\rm{std}})$.

We can use the full generality of Theorem~\ref{rationalmaintheorem} to show that sufficiently large inadmissible transverse surgery on a (rationally) null-homologous fibred knot $K \subset (M, \xi_K)$ preserves the non-vanishing of $\cplusred$.

\begin{cor}\label{cplusred preserved}
Let $K \subset M$ be a (rationally) null-homologous fibred knot, and fix a framing of $K$.  If $\cplusred(\xi_K) \neq 0$, then $\cplusred(\xi_{K_r}) \neq 0$ for all sufficiently large $r$ (measured with respect to the fixed framing).
\end{cor}

As an amusing application, Theorem~\ref{maintheorem} gives many examples of rationally fibred knots $K$ such that both $K$ and $\overline{K}$ support tight contact structures.  The first such examples were discovered by Hedden and Plamenevskaya in \cite[Corollary~2]{HP}. This is a phenomenon that for null-homologous fibred knots only occurs when the monodromy of the fibration is isotopic to the identity.  To create these examples, start with a null-homologous fibred genus-$g$ knot $K'$ that supports a tight contact structure $\xi_{K'}$ on $M'$, and such that $\cplusred(\xi_{K'}) = 0$.  By \cite[Theorem~1.6]{Conway}, inadmissible transverse $r$-surgery on $K'$ is tight for $r > 2g-1$.  Thus, if $K = K'_r$ is the surgery dual knot to $K'$ under this $r$-surgery, then both $K$ and $\overline{K} = \overline{K'}_{-r}$ support tight contact structures, by \cite[Theorem~1.6]{Conway} and Theorem~\ref{maintheorem}, respectively.

Up to this point, we have only considered negative surgery coefficients.  This is because admissible transverse surgery is only defined for slopes in an interval $[-\infty, a)$, where $a$ is often hard to pin down.  However, when the fibred knot $K \subset M$ supports an overtwisted contact structure $\xi_K$, and the associated monodromy is not right-veering (see \cite{HKM:right1} for details), we can show that admissible transverse surgery is defined for any positive surgery coefficient, and that such a surgery is tight.

\begin{thm}\label{positive surgeries}
Let $K \subset M$ be a null-homologous fibred knot supporting an overtwisted contact structure $\xi_K$.  Assume further that the monodromy $\phi_K$ of the open book associated to the fibration is not right-veering.  Then we can define admissible transverse $r$-surgery for all $r \in \Q$, and the resulting manifold $(M_r(K),\xi_r)$ satisfies $\cplusred(\xi_r) \neq 0$ for all $r \geq 0$ (in addition to whatever is guaranteed by Theorem~\ref{maintheorem}).
\end{thm}

\begin{remark}\label{admiss vs inadmiss}
For $r > 0$, we stress that $\xi_r$ given in Theorem~\ref{positive surgeries} is not the contact structure $\xi_{K_r}$ on $M_r(K)$ supported by the rationally fibred knot $K_r$.  Indeed, $\xi_r$ is the result of admissible transverse surgery on $K$, whereas $\xi_{K_r}$ for $r > 0$ is the result of inadmissible transverse surgery on $K$ (see Theorem~\ref{thm:OBD surgery}).
\end{remark}

We can use Theorems~\ref{rationalmaintheorem}~and~\ref{positive surgeries} to narrow down the possibilities for manifolds that do not support tight contact structures.

\begin{cor}\label{where surgery is not tight}
If $\HFred(M) = 0$, and $K \subset M$ is a null-homologous fibred knot, then if $M_r(K)$ does not support a tight contact structure, then $\phi_K$ is right-veering and $r > 0$.

For a general $M$ and $K \subset M$ is null-homologous and fibred, there exists at most one $r \in \Q\cup \{\infty\}$ such that both orientations of $M_r(K)$ do not support a tight contact structure.  In addition, at least one of the following is true:
\begin{itemize}
\item $M_r(K)$ supports a tight contact structure in both of its orientations for all $r > 0$.
\item $M_r(K)$ supports a tight contact structure in both of its orientations for all $r < 0$.
\end{itemize}
\end{cor}

This paper is organised as follows: in Section~\ref{sec:background} we recall some Heegaard Floer homology, transverse surgery, and how they interact; in Section~\ref{sec:neg surgeries}, we deal with the proof of Theorem~\ref{maintheorem} and Corollaries~\ref{FDTC > 1 tight}--\ref{cplusred preserved}; in Section~\ref{sec:pos surgeries}, we prove Theorem~\ref{positive surgeries} and Corollary~\ref{where surgery is not tight}; and in Section~\ref{sec:questions}, we look at some further questions, namely, the difference between $\cplus$ and $\chat$, and genus-1 fibred knots.

\subsection*{Acknowledgements}

The author would like to thank John Etnyre for many helpful discussion throughout this project.  Additionally, the author has benefited from discussions with and/or comments on an early draft of this paper from David Shea Vela-Vick, Tom Mark, B\"{u}lent Tosun, Matt Hedden, and an anonymous referee.  This work was partially supported by NSF Grant DMS-13909073.

\section{Background}
\label{sec:background}

We assume familiarity with the basics of Heegaard Floer Homology and knot Floer Homology (\cite{OS:hf1, OS:hfk}); Legendrian and transverse knots, surgery on Legendrian knots, open book decompositions, and convex surfaces (\cite{Etnyre:knots, Etnyre:contactlectures, Etnyre:OBlectures, DGS:surgery}).  Here, we will recall the relevant information.

\subsection{Heegaard Floer Homology}
\label{subsec:HF}

All of the Heegaard Floer groups in this paper will be defined over the field $\F = \Z/2\Z$.  As per convention, given a generating set $\{\bm{x}\}$ for $\CFhat(M)$, we consider $\CFinfinity(M)$ as being generated by $\{U^i\bm{x}\,|\,i \in \Z\}$.  Then, $\CFminus(M)$ is generated by $\{U^i\bm{x}\,|\,i \in \Z_{> 0}\}$, and $\CFplus(M) = \CFinfinity(M) / \CFminus(M)$.  These chain complexes are involved in several short exact sequences, leading to homology long exact sequences.

$$\cdots \xrightarrow{\delta_+} \HFhat(M) \xrightarrow{\widehat{\iota}} \HFplus(M) \xrightarrow{U} \HFplus(M) \xrightarrow{\delta_+} \cdots$$

$$\cdots \xrightarrow{\delta_-} \HFminus(M) \xrightarrow{U} \HFminus(M) \xrightarrow{\pi_-} \HFhat(M) \xrightarrow{\delta_-} \cdots$$

$$\cdots \xrightarrow{\delta_\infty} \HFminus(M) \xrightarrow{\iota_-} \HFinfinity(M) \xrightarrow{\pi_\infty} \HFplus(M) \xrightarrow{\delta_\infty} \cdots$$

For the second short exact sequence, we naturally identify the chain complex generated by $\{U\bm{x}\}$ with $\CFhat(M)$.  With this identification, the commutativity of the following diagram can be checked on the chain level.

\begin{center}
\begin{tikzpicture}[node distance=2cm,auto]
  \node (UL) {$\HFhat(M)$};
  \node (UM) [right of=UL] {\null};
  \node (UR) [right of=UM] {$\HFplus(M)$};
  \node (BR) [below of=UM] {$\HFminus(M)$.};
  \draw[->] (UL) to node {$\widehat\iota$} (UR);
  \draw[->] (UR) to node {$\delta_\infty$} (BR);
  \draw[->] (UL) to node[below]{$\delta_-$} (BR);
\end{tikzpicture}
\end{center}

We define $\HFred(M)$ to be $\coker \pi_\infty$, or equivalently $\ker \iota_-$; these two definitions are isomorphic via $\delta_\infty$.

The Heegaard Floer chain groups for $-M$ are dual to those for $M$.  In particular, we get a natural non-degenerate bilinear pairing (the Kronecker pairing) on $\CFhat(-M) \otimes \CFhat(M)$ that descends to homology:
$$\langle \cdot, \cdot \rangle : \HFhat(-M) \otimes \HFhat(M) \to \F.$$

There is a similar pairing for $\HFpm(-M)$ and $\HFmp(M)$, see \cite{OS:4manifolds1} for details.  For our purposes, it is sufficient to note that if $\delta_-(x) \neq 0$, where $x \neq 0 \in \HFhat(-M)$ and $\delta_- : \HFhat(-M) \to \HFplus(-M)$, then there exists some $y \neq 0 \in \HFhat(M)$ such that $\langle x, y \rangle \neq 0$ and $y \in \im \delta_+$, or equivalently, $\widehat\iota(y) = 0 \in \HFplus(M)$.

Putting this together with the fact that $\delta_- = \delta_\infty \circ \widehat\iota$, we get the following characterisation of elements of $\HFred(-M)$, which we will use in the proof of Theorem~\ref{maintheorem}.

\begin{lem}\label{HF duality}
Given an element $x \in \HFhat(-M)$, then $\widehat\iota(x) \in \HFplus(-M)$ descends to a non-zero element of $\HFred(-M)$ if and only if there exists some $y \neq 0 \in \HFhat(M)$ such that $\langle x, y \rangle \neq 0$ and $\widehat\iota(y) = 0 \in \HFplus(M)$.
\end{lem}

\subsection{Knot Heegaard Floer Homology}
\label{subsec:HFK}

Ozsváth and Szabó \cite{OS:hfk} and independently Rasmussen \cite{Rasmussen} define a filtration on $\CFhat(M)$ induced by a (rationally) null-homologous knot $K \subset M$; this filtration $\mathcal{A}$ is called the {\it Alexander grading}.  We get an induced bi-filtration on $\CFinfinity(M)$ given by $\mathcal{F}(U^i\bm{x}) = (-i, \mathcal{A}(\bm{x}) - i)$.  We will use $\mathcal{C}_K$ to denote the complex $\CFinfinity(M)$ filtred by $K$, and as per convention, $\mathcal{C}_K\{\mbox{relation}(i,j)\}$ will denote sub and quotient complexes of $\mathcal{C}_K$ determined by $\mbox{relation}(i,j)$.

\begin{remark}
The Alexander grading depends on a choice of relative Spin$^c$ structure $\spins \in \Spinc(M,K)$.  However, different choices of $\spins$ give isomorphic chain complexes, and in practice amounts to shifting the Alexander grading by a fixed constant.  Since we are only interested in the structure of $\mathcal{C}_K$ and not the particular values of the Alexander grading, we need not worry about this issue.  For the remainder of the article, we assume that we have fixed $\spins$, whenever needed.  See \cite{OS:hfk, OS:rationalsurgery, HP, Raoux} for more details.
\end{remark}

First note that $\mathcal{C}_K\{i = 0\} = \CFhat(M)$, and $\mathcal{C}_K\{i \geq 0\} = \CFplus(M)$.  The homology of quotient complexes of $\mathcal{C}_K\{i = 0\}$ tell us whether $K$ is fibred.

\begin{thm}[\cite{Ni}, \cite{HP}]\label{thm:fibred HFK}
A rationally null-homologous knot $K\subset M$ with irreducible complement is fibred if and only if $H_*(\mathcal{C}_K\{i = 0, j = {\rm bottom}\})\cong\F$, where ${\rm bottom}$ is the smallest value of $j$ such that the homology is non-trivial.  If $K$ is integrally null-homologous, then ${\rm bottom} = g(K)$, where $g(K)$ is the genus of the knot.
\end{thm}

Let $K$ be a rationally null-homologous knot with a fixed framing.  If $K$ is null-homologous, then we can choose the Seifert framing, and if not, then we can choose the \textit{canonical framing} (see \cite{MT, Raoux}), but this choice will not affect our results.

According to \cite{OS:hfk, OS:rationalsurgery}, we can use $\mathcal{C}_K$ to determine the Heegaard Floer homology of sufficiently large (positive or negative) surgeries on $K$ (with respect to the fixed framing).  For our purposes, it will be sufficient to identify the complexes $$\mathcal{C}_{K_n}\{i = 0, j = {\rm bottom}\} \mbox{, }\mathcal{C}_{K_n}\{i = 0\} \mbox{, and } \mathcal{C}_{K_n}\{i \geq 0\},$$ where $K_n$ is the surgery dual knot to $K$ after $n$-surgery, for a large integer $n > 0$.  We will also determine $$\mathcal{C}_{K_{-n}}\{i = 0, j = {\rm top}\} \mbox{, } \mathcal{C}_{K_{-n}}\{i = 0\} \mbox{, and } \mathcal{C}_{K_{-n}}\{i \geq 0\},$$ where $K_{-n}$ is the surgery dual knot to $K$ after $-n$-surgery, for a large integer $n > 0$.

Let ${\rm top}$ be the maximum value of $j_0$ such that $H_*\left(\mathcal{C}_K\{i = 0, j = j_0\}\right)$ is non-trivial, and define ${\rm bottom}$ as the minimum such value.  Let $\widehat{A} = \mathcal{C}_{K}\{\max(i,j-{\rm top}+1)=0\}$, let $A^+ = \mathcal{C}_{K}\{\max(i,j-{\rm top}+1)\geq0\}$, and let $S = \mathcal{C}_{K}\{i < 0 , j = {\rm top}-1\} = \mathcal{C}_K\{i = -1, j = {\rm top}-1\}$.  The last equality holds because for a fixed $i_0$, we can assume that $\mathcal{C}_K\{i = i_0\}$ is non-trivial only for ${\rm bottom} < j - i_0 < {\rm top}$.

\begin{thm}[\cite{OS:hfk, OS:rationalsurgery, HP}]\label{thm:HF positive surgery formula}
Let $K \subset M$ be a rationally null-homologous knot with a fixed framing.  For all sufficiently large integers $n$, we have:
\begin{itemize}
\item $\mathcal{C}_{K_n}\{i = 0, j = {\rm bottom}\} \simeq S$,
\item $\mathcal{C}_{K_n}\{i = 0\} = \CFhat(M_n(K)) \simeq \widehat{A}$,
\item $\mathcal{C}_{K_n}\{i \geq 0\} = \CFplus(M_n(K)) \simeq A^+$,
\item $\chat(\xi_{K_n})$ is the image of $H_*(S) \to H_*(A)$.
\end{itemize}
\end{thm}

Let $\widehat{A}' = \mathcal{C}_{K}\{\min(i,j - {\rm bottom} - 1)=0\}$, let $A'^+ = \mathcal{C}_{K}\{\min(i,j-{\rm bottom}-1)\geq0\}$, and let $Q' = \mathcal{C}_{K}\{i > 0 , j = {\rm bottom}+1\} = \mathcal{C}_K\{i = 1, j = {\rm bottom} + 1\}$.

\begin{thm}[\cite{OS:hfk, OS:rationalsurgery, HP}]\label{thm:HF negative surgery formula}
Let $K \subset M$ be a rationally null-homologous knot with a fixed framing.  For all sufficiently large integers $n$, we have:
\begin{itemize}
\item $\mathcal{C}_{K_{-n}}\{i = 0, j = {\rm top}\} \simeq Q'$,
\item $\mathcal{C}_{K_{-n}}\{i = 0\} = \CFhat(M_{-n}(K)) \simeq \widehat{A}'$,
\item $\mathcal{C}_{K_{-n}}\{i \geq 0\} = \CFplus(M_{-n}(K)) \simeq A'^+$.
\end{itemize}
\end{thm}

\begin{remark}
The proof of \cite[Theorem~4.2]{HP}, on which the first bullet points of both the theorems above are based, goes through with no changes for rationally null-homologous fibred knots, after taking care when calculating the values of the Alexander filtration.
\end{remark}

Based on Theorem~\ref{thm:fibred HFK}, we conclude that if $K$ is fibred, $H_*(Q')$ and $H_*(S)$ are singly generated.

\subsection{Transverse Surgery and Open Books}
\label{subsec:transverse surgery}

Consider a contact manifold $(M, \xi)$ containing a transverse knot $K$.  Fix a framing on $K$, against which all slopes will be measured.  There is a natural analogue of Dehn surgery on transverse knots, called {\it transverse surgery}.  This comes in two flavours: {\it admissible} and {\it inadmissible}.  See \cite{BE:transverse, Conway} for more details.

Given a transverse knot $K\subset (M, \xi)$, a neighbourhood of $K$ is contactomorphic to $\{r \leq a\} \subset (S^1 \times \R^2, \xi_{\rm{rot}}=\ker(\cos r\, dz + r\sin r\,d\theta))$ for some $a$, where $z$ is the coordinate on $S^1$, and $K$ is identified with the $z$-axis.  The slope of the leaves of the foliation induced by the contact planes on the torus $\{r = r_0\}$ is $-\cot r_0 / r_0$.  We will abuse notation and refer to the torus $\{r = r_0\}$ (resp.{}\ the solid torus $\{r \leq r_0\}$) as being a torus (resp.{}\ a solid torus) of slope $-\cot r_0 / r_0$.

If $L$ is a Legendrian approximation of $K$, then inside a standard neighbourhood of $L$, we can find a neighbourhood of $K$ where the boundary torus is of slope $p/q$, for any $p/q < tb(L)$.  Conversely, if we have a neighbourhood for $K$ where the boundary torus is of slope $p/q$, then there exist Legendrian approximations $L_n$ of $K$ with $tb(L_n) = n$, for any integer $n < p/q$.

To perform {\it admissible transverse surgery}, we take a torus of rational slope $p/q$ inside $\{r \leq a\}$, we remove the interior of the corresponding solid torus of slope $p/q$, and perform a contact cut on the boundary, {\it ie.\ }we quotient the boundary by the $S^1$-action of translation along the leaves of slope $p/q$ (see \cite{Lerman} for details); the contact structure descends to the quotient manifold to give the result of admissible transverse $p/q$-surgery on $K$.  To perform {\it inadmissible transverse surgery}, we first remove $\{r < b\}$ for some $b \leq a$. We then glue on the $T^2 \times I$ layer $\{r_0 \leq r \leq b + 2\pi\}$ by identifying $\{r = b + 2\pi\}$ with $\bd\left( M\backslash \{r < b\}\right)$; we choose $r_0$ such that the new boundary $\{r = r_0\}$ is a torus of slope $p/q$.  The result of inadmissible transverse $p/q$-surgery on $K$ is the result of performing a contact cut on the new boundary.

A (rationally) fibred link $L \subset M$ induces an open book decomposition of $M$, which supports a unique contact structure $\xi_L$ and in which $L$ is a transverse link (\cite{Giroux:OBD, BEVHM}).  The construction of $\xi_L$ gives a natural neighbourhood of each component of $L$ in which we can define admissible transverse surgery up to the slope given by the page (the \textit{page slope}).

\begin{remark}\label{multiple tori}
As in \cite[Remark 2.9]{Conway}, the result of admissible transverse surgery depends in general on the neighbourhood of $K$ chosen to define it.  For fibred knots $K \subset (M, \xi_K)$, we will always be using the neighbourhood mentioned above, unless otherwise noted.
\end{remark}

If a link $L\subset M$ is (rationally) fibred, then topological $r$-surgery on a component $K$ of $L$ for any $r$ not equal to the page slope gives a rationally fibred knot $L' \subset M'$.

\begin{thm}[\cite{BEVHM}, \!\cite{Conway}]\label{thm:OBD surgery} The contact structure $\xi_{L'}$ on $M'$ is the result of admissible (resp.{}\ inadmissible) transverse $r$-surgery on the knot $K \subset (M, \xi_L)$, when $r$ is less than (resp.{}\ greater than) the page slope. \end{thm}

In \cite[Section~3.3]{Conway}, the author gave an algorithm to exhibit an abstract integral open book supporting $\xi_{L'}$.  To do this, we break up an admissible transverse surgery into a sequence of integral admissible transverse surgeries, and an inadmissible transverse surgery is treated as an inadmissible transverse $(f+1/n)$-surgery followed by an admissible transverse surgery (where $f$ is the page slope).  The algorithm can be described in terms of the following two moves: (1) adding a boundary parallel Dehn twist about a boundary component of the current open book, and (2) stabilising the open book along a boundary parallel arc.  The first move is an integral admissible (resp.{} inadmissible) transverse surgery if we add a positive (resp.{} negative) Dehn twist, and the second move changes the surgery coefficient of subsequent surgeries.  See \cite[Section~3.3]{Conway} for the details.

\subsection{Heegaard Floer Contact Elements}
\label{subsec:HF and contact geometry}

As in \cite{OS:contact, HP}, when $K \subset M$ is rationally fibred, the \textit{contact element} $\widehat{c}(\xi_K) \in \HFhat(-M)$ of the supported contact structure is given by the image of
$$H_*(\mathcal{C}_{\overline{K}}\{i = 0, j = {\rm bottom}\}) \to H_*(\mathcal{C}_{\overline{K}}\{i = 0\}) \cong \HFhat(-M),$$ where $\overline{K}$ is the knot $K$ considered as a subset of $-M$.  We denote by $\cplus(\xi_K)$ and $\cplusred(\xi_K)$ the elements $\widehat\iota(\chat(\xi_K)) \in \HFplus(-M)$ and $\delta_\infty(\cplus(\xi_K)) \in \HFred(-M)$.  By Lemma~\ref{HF duality}, $\cplusred(\xi_K) \neq 0$ if and only if there exists some $y \neq 0 \in \HFhat(M)$ such that $\langle \chat(\xi_K), y \rangle \neq 0$ and $\widehat\iota(y) = 0 \in \HFplus(M)$.

The contact elements interest us because their non-vanishing proves that the contact structure is tight (see \cite{OS:contact}).  Furthermore, they behave well under capping off binding components of open books (\cite{Baldwin:cappingoff}) and admissible transverse surgery, as in the following lemma.

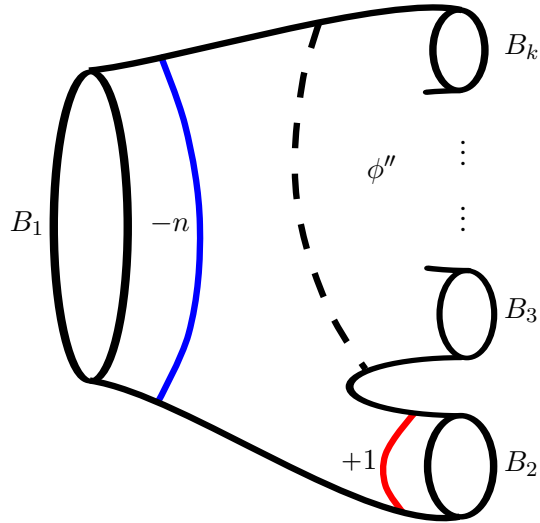
\begin{figure}[htbp]
\begin{center}
\begin{overpic}[scale=3,tics=20]{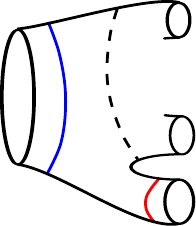}
\put(-15,110){\large $B_1$}
\put(172,19){\large $B_2$}
\put(172,77){\large $B_3$}
\put(172,177){\large $B_k$}
\put(110,20){\large $+1$}
\put(38,110){\large $-n$}
\put(155,110){\large $\vdots$}
\put(155,135){\large $\vdots$}
\put(120,130){\large $\phi''$}
\end{overpic}
\caption{This open book is glued along $B_1$ to do transverse inadmissible $s$-surgery, for $1/n < s < 1/(n-1)$.  The monodromy is $\tau_{B_1}^{-n}\tau_{B_2}\phi''$, where the support of $\phi''$ is to the right of the dashed curve.}
\label{fig:positive surgery}
\end{center}
\end{figure}

\begin{lem}\label{lemma:admissible preserves non-vanishing}
Let $K \subset (M, \xi)$ be a transverse knot (not necessarily fibred or null-homologous).  If $\cplus(\xi) \neq 0$ (resp.{}\ $\cplusred(\xi) \neq 0$), then the result $(M',\xi')$ of admissible transverse surgery on $K$ satisfies $\cplus(\xi')\neq0$ (resp.{}\ $\cplusred(\xi') \neq 0$).
\end{lem}
\begin{proof}
If $K' \subset (M', \xi')$ is the surgery dual knot to $K$, then some inadmissible transverse $s$-surgery on $K'$ results in $(M, \xi)$, where $s$ is measured with respect to some framing of $K'$.  By \cite[Lemma~6.5]{BEVHM}, there exists an open book decomposition for $(M',\xi')$ such that $K'$ is a component of the binding.  We then stabilise this open book along arcs parallel to the boundary component corresponding to $K'$ to get an open book $(\Sigma', \phi')$, such that $s > 0$ when measured against the framing induced by $\Sigma'$ on $K'$.  By Theorem~\ref{thm:OBD surgery}, we can construct an integral open book for $(M, \xi)$ from $(\Sigma', \phi')$, using the algorithm from \cite[Section~3.3]{Conway}.

Let $0$ be the page slope on $K'$.  If $s = 1/n$ for some positive integer $n$, then $(M, \xi)$ is supported by the open book $(\Sigma', \phi'\circ\tau_{B'}^{-n})$, where $\tau_{B'}$ is a Dehn twist about the boundary component $B'$ corresponding to $K'$.  By \cite{BEVHM, Baldwin:monoid}, the property that a monodromy supports a contact structure with non-vanishing contact invariant is preserved under composition.  Hence, since $(\Sigma', \tau_{B'}^n)$ supports a contact structure with non-vanishing contact invariant, and $\phi' = \left(\phi'\circ\tau_{B'}^{-n}\right)\circ\tau_{B'}^n$, so too does $(\Sigma', \phi')$ support a contact structure with non-vanishing contact invariant.

If $1/n < s < 1/(n-1)$ (where if $n=1$, we mean just $s > 1$), then by \cite[Section~3.3]{Conway}, $(M, \xi)$ is supported by an open book $(\Sigma' \cup \Sigma'', \phi'\circ\tau^{-n}_{B_1}\tau_{B_2}\phi'')$.  Here, the open book $(\Sigma'',\tau_{B_1}^{-n}\tau_{B_2}\phi'')$ has boundary components $B_1, B_2,\ldots, B_k$ (where $k \geq 3$), and $\phi''$ is supported in a neighbourhood of $B_3 \cup \cdots \cup B_k$ (see Figure~\ref{fig:positive surgery}); we glue $\Sigma'$ to $\Sigma''$ by identifying $B_1$ with $B'$ (the binding component corresponding to $K'$).  In this open book for $(M, \xi)$, if we cap off $B_3,\ldots,B_k$, we arrive at an open book $(\Sigma', \phi'\circ\tau_{B'}^{-n+1})$.

Since by \cite{Baldwin:cappingoff}, capping off an open book preserves the non-vanishing of the contact class, if $n = 1$, then we are done.  If $n > 1$, then $-n+1 > 0$, and we have an open book for inadmissible transverse $1/(-n+1)$-surgery on $K'$, where the supported contact structure has non-vanishing contact invariant.  We finish the proof as in the case $s = 1/(-n+1)$.
\end{proof}

\section{Negative Surgeries}
\label{sec:neg surgeries}

In this section, we will state and prove Theorem~\ref{maintheorem} in full generality (as Theorem~\ref{rationalmaintheorem}), and then prove its corollaries.

\begin{thm}\label{rationalmaintheorem}
Let $K \subset M$ be a rationally null-homologous fibred knot with a fixed framing, and let $f$ be the slope on $\bd N(K)$ induced by a page of the fibration.  $(M_r(K), \xi_{K_r})$ is the result of admissible transverse $r$-surgery on $K \subset (M, \xi_K)$, where $r$ is measured with respect to the fixed framing, then:
\begin{enumerate}
\item $c^+(\xi_{K_r}) \neq 0$ for all $r < f$ if and only if $c^+_{\rm{red}}(\xi_{\overline{K}}) = 0$.
\item $c^+_{\rm{red}}(\xi_{K_r}) \neq 0$ for all $r < f$ if and only if $c^+(\xi_{\overline{K}}) = 0$.
\end{enumerate}
\end{thm}
\begin{proof}
To keep notation simple, we will assume $K$ is null-homologous, where the fixed framing is taken to be the Seifert framing (or equivalently, the page slope $f$), and we fix a relative Spin$^c$ structure to measure the Alexander grading on $\mathcal{C}_K$ and $\mathcal{C}_{\overline{K}}$ such that ${\rm top} = g$ and ${\rm bottom} = -g$, where $g = g(K)$ is the genus of $K$.  The proof for the more general case is the same, \textit{mutatis mutandis}.

It follows from Lemma~\ref{lemma:admissible preserves non-vanishing} that the vanishing of $\cplus(\xi_{K_r})$ (resp.{}\ $\cplusred(\xi_{K_r})$) for some $r < 0$ implies the same vanishing for $\xi_{K_n}$, where $n$ is a sufficiently negative integer.

Thus, to prove Theorem~\ref{rationalmaintheorem} we only need to show that the theorem is true when $r$ is restricted to be a sufficiently negative integer $n$, using Theorems~\ref{thm:HF positive surgery formula}~and~\ref{thm:HF negative surgery formula}.

\subsection*{Proof of (1)} Since $\cplus(\xi_{K_n}) \in \HFplus(-\left(M_n(K)\right)) = \HFplus(\left(-M\right)_{-n}(\overline{K}))$, we will look at $\mathcal{C}_{\overline{K}}$, and large positive integer surgeries on $\overline{K}$.  We start by noting two short exact sequences (where the left map is inclusion), which commute with the obvious vertical inclusions:

\begin{center}
\begin{tikzpicture}[node distance=3cm,auto]
  \node (UL) {$\mathcal{C}_{\overline{K}}\{i = 0, j \leq g-1\}$};
  \node (ULL) [left of=UL, node distance = 3cm] {$0$};
  \node (UM) [right of=UL, node distance=4cm] {$\mathcal{C}_{\overline{K}}\{i = 0\}$};
  \node (UR) [right of=UM, node distance=4cm] {$\mathcal{C}_{\overline{K}}\{i = 0, j = g\}$};
  \node (URR) [right of=UR, node distance = 3cm] {$0$};
  \node (LL) [below of=UL, node distance = 2cm] {$\mathcal{C}_{\overline{K}}\{i = 0, j \leq g-1\}$};
  \node (LLL) [left of=LL, node distance = 3cm] {$0$};
  \node (LM) [right of=LL, node distance=4cm] {$\mathcal{C}_{\overline{K}}\{i \geq 0\}$};
  \node (LR) [right of=LM, node distance=4cm] {$\mathcal{C}_{\overline{K}}\{\max(i,j-g+1) \geq 1\}$};
  \node (LRR) [right of=LR, node distance = 3cm] {$0$};
  \draw[->] (ULL) to (UL);
  \draw[->] (UL) to (UM);
  \draw[->] (UM) to (UR);
  \draw[->] (UR) to (URR);
  \draw[->] (LLL) to (LL);
  \draw[->] (LL) to (LM);
  \draw[->] (LM) to (LR);
  \draw[->] (LR) to (LRR);
  \draw[->] (UL) to (LL);
  \draw[->] (UM) to (LM);
  \draw[->] (UR) to (LR);
\end{tikzpicture}
\end{center}

Denote the left-hand complexes by $D$.  We know that $\mathcal{C}_{\overline{K}}\{i = 0 \} = \CFhat(-M)$ and $\mathcal{C}_{\overline{K}}\{i \geq 0\} = \CFhat(-M)$.  The two right-hand complexes can be identified via $U$ with $\mathcal{C}_{\overline{K}}\{i = -1, j = g-1\} = S$ and $\mathcal{C}_{\overline{K}}\{\max(i, j - g + 1) \geq 0\} = A^+$.  The following commutative diagram comes from the homology long exact sequence associated to the above diagram.

\begin{center}
\begin{tikzpicture}[node distance=3cm,auto]
  \node (UL) {$H_*(D)$};
  \node (UM) [right of=UL, node distance=3cm] {$\HFhat(-M)$};
  \node (UR) [right of=UM, node distance=3cm] {$H_*(S)$};
  \node (URR) [right of=UR, node distance = 3cm] {$H_*(D)$};
  \node (LL) [below of=UL, node distance = 2cm] {$H_*(D)$};
  \node (LM) [right of=LL, node distance=3cm] {$\HFplus(-M)$};
  \node (LR) [right of=LM, node distance=3cm] {$H_*(A^+)$};
  \node (LRR) [right of=LR, node distance = 3cm] {$H_*(D)$};
  \draw[->] (UL) to node{$\widehat\iota_D$} (UM);
  \draw[->] (UM) to node{$\widehat\pi$} (UR);
  \draw[->] (UR) to node{$\widehat\partial$} (URR);
  \draw[->] (LL) to node{$\iota^+_D$} (LM);
  \draw[->] (LM) to node{$\pi^+$} (LR);
  \draw[->] (LR) to node{$\partial^+$} (LRR);
  \draw[->] (UL) to node{$\cong$} (LL);
  \draw[->] (UM) to node{$\widehat\iota$} (LM);
  \draw[->] (UR) to node{$\iota_S$} (LR);
  \draw[->] (URR) to node{$\cong$} (LRR);
\end{tikzpicture}
\end{center}

By Theorem~\ref{thm:fibred HFK}, $H_*(S)$ is generated by a single element $x$.  The quotient complex $\mathcal{C}_{\overline{K}}\{i = 0, j = g\}$ is dual to the subcomplex $\mathcal{C}_K\{i = 0, j = -g\}$, and $\chat(\xi_{\overline{K}})$ is the image of $$H_*(\mathcal{C}_K\{i = 0, j = -g\}) \to \HFhat(M).$$ It follows that if $\chat(\xi_{\overline{K}}) \neq 0$, then any element $y \in \widehat\pi^{-1}(x) \subset \HFhat(-M)$ will pair non-trivially with $\chat(\xi_{\overline{K}})$. Thus, based on Theorem~\ref{thm:HF positive surgery formula} and the discussion in Section~\ref{subsec:HF and contact geometry}, we deduce the following facts:

\begin{itemize}
\item $\cplusred(\xi_{\overline{K}}) \neq 0$ if and only if $\widehat\partial(x) = 0$ and there is a $y \in \widehat\pi^{-1}(x)$ such that $\widehat\iota(y) = 0$;
\item $\cplus(\xi_{K_n}) = 0$ for sufficiently negative integers $n$ if and only if $\iota_S(x) = 0$.
\end{itemize}

First assume that $\cplusred(\xi_{\overline{K}}) \neq 0$.  As above, we can find some non-zero $y \in \HFhat(-M)$ such that $\widehat\pi(y) = x$ and $\widehat\iota(y) = 0$.  But then by commutativity, $\iota_S(x) = \pi^+\widehat\iota(y) = 0$, and so $\cplus(\xi_{K_n}) = 0$.

Now assume that $\cplus(\xi_{K_n}) = 0$.  Thus, $\iota_S(x) = 0$.  Since $\widehat\partial(x) = \partial^+\iota_S(x) = 0$, we know that $\chat(\xi_{\overline{K}}) \neq 0$.  Then, pick a non-zero element $y_0 \in \widehat\pi^{-1}(x)$.  If $\widehat\iota(y_0) = 0$, then we're done, so assume that $\widehat\iota(y_0) \neq 0$.  Since $\pi^+(\widehat\iota(y_0)) = \iota_S(x) = 0$, we can find some element $d \in H_*(D)$ such that $\iota^+_D(d) = \widehat\iota(y_0)$.  But now, $y_0 - \widehat\iota_D(d) \in \widehat\pi^{-1}(x)$, and $\widehat\iota(y_0 - \widehat\iota_D(d)) = \widehat\iota(y_0) - \iota^+_D(d) = 0$.  Thus, $\cplusred(\xi_{\overline{K}}) \neq 0$.

\subsection*{Proof of (2)} This time, we will consider $\mathcal{C}_K$, and sufficiently negative integer surgeries on $K$.  As in (1), we start by noting two short exact sequences (where the left map is inclusion), which commute with the obvious vertical inclusions:

\begin{center}
\begin{tikzpicture}[node distance=3cm,auto]
  \node (UL) {$\mathcal{C}_K\{i = 0, j \geq -g+1\}$};
  \node (ULL) [left of=UL, node distance = 3cm] {$0$};
  \node (UM) [right of=UL, node distance=5cm] {$\mathcal{C}_K\{\max(i, j + g - 1) = 0\}$};
  \node (UR) [right of=UM, node distance=5cm] {$\mathcal{C}_K\{i \geq 1, j = -g+1\}$};
  \node (URR) [right of=UR, node distance = 3cm] {$0$};
  \node (LL) [below of=UL, node distance = 2cm] {$\mathcal{C}_K\{i = 0, j \geq -g+1\}$};
  \node (LLL) [left of=LL, node distance = 3cm] {$0$};
  \node (LM) [right of=LL, node distance=5cm] {$\mathcal{C}_K\{\max(i, j + g - 1) \geq 0\}$};
  \node (LR) [right of=LM, node distance=5cm] {$\mathcal{C}_K\{i \geq 1\}$};
  \node (LRR) [right of=LR, node distance = 3cm] {$0$};
  \draw[->] (ULL) to (UL);
  \draw[->] (UL) to (UM);
  \draw[->] (UM) to (UR);
  \draw[->] (UR) to (URR);
  \draw[->] (LLL) to (LL);
  \draw[->] (LL) to (LM);
  \draw[->] (LM) to (LR);
  \draw[->] (LR) to (LRR);
  \draw[->] (UL) to (LL);
  \draw[->] (UM) to (LM);
  \draw[->] (UR) to (LR);
\end{tikzpicture}
\end{center}

Denote the left-hand complexes by $D$.  The two middle complexes are $\widehat{A}'$ and $A'^+$, and the top-right complex is $Q'$ (recall, $\mathcal{C}_K\{i > 1, j = -g+1\}$ is empty).  The two right-hand complexes can also be identified via $U$ with $\mathcal{C}_K\{i = 0, j = -g\}$ and $\CFplus(M)$.  The following commutative diagram comes from the homology long exact sequence associated to the above diagram.

\begin{center}
\begin{tikzpicture}[node distance=3cm,auto]
  \node (UL) {$H_*(D)$};
  \node (UM) [right of=UL, node distance=3cm] {$H_*(\widehat{A}')$};
  \node (UR) [right of=UM, node distance=3cm] {$H_*(Q')$};
  \node (URR) [right of=UR, node distance = 3cm] {$H_*(D)$};
  \node (LL) [below of=UL, node distance = 2cm] {$H_*(D)$};
  \node (LM) [right of=LL, node distance=3cm] {$H_*(A'^+)$};
  \node (LR) [right of=LM, node distance=3cm] {$\HFplus(M)$};
  \node (LRR) [right of=LR, node distance = 3cm] {$H_*(D)$};
  \draw[->] (UL) to node{$\widehat\iota_D$} (UM);
  \draw[->] (UM) to node{$\widehat\pi$} (UR);
  \draw[->] (UR) to node{$\widehat\partial$} (URR);
  \draw[->] (LL) to node{$\iota^+_D$} (LM);
  \draw[->] (LM) to node{$\pi^+$} (LR);
  \draw[->] (LR) to node{$\partial^+$} (LRR);
  \draw[->] (UL) to node{$\cong$} (LL);
  \draw[->] (UM) to node{$\iota_A$} (LM);
  \draw[->] (UR) to node{$\iota_Q$} (LR);
  \draw[->] (URR) to node{$\cong$} (LRR);
\end{tikzpicture}
\end{center}

By Theorem~\ref{thm:fibred HFK}, $H_*(Q')$ is generated by a single element $x$.  The complex $Q' = \mathcal{C}_{K_n}\{i = 0, j = \rm{top}\}$ is dual to $S = \mathcal{C}_{\overline{K}_{-n}}\{i = 0, j = \rm{bottom}\}$, and so if $\chat(\xi_{K_n}) \neq 0$, then any element $y \in \widehat\pi^{-1}(x) \subset H_*(A')$ will pair non-trivially with $\chat(\xi_{K_n})$.   Also, since $Q'$ can be identified via $U$ with $\mathcal{C}_K\{i = 0, j = -g\}$, the map $\iota_Q$ can be identified with the map $$H_*(\mathcal{C}\{i = 0, j = -g\}) \to \HFhat(M) \to \HFplus(M).$$  Thus, based on Theorem~\ref{thm:HF negative surgery formula} and the discussion in Section~\ref{subsec:HF and contact geometry}, we deduce the following facts:

\begin{itemize}
\item $\cplusred(\xi_{\overline{K}_{n}}) \neq 0$ for sufficiently negative integers $n$ if and only if $\widehat\partial(x) = 0$ and there is a $y \in \widehat\pi^{-1}(x)$ such that $\iota_A(y) = 0$;
\item $\cplus(\xi_{\overline{K}}) = 0$ if and only if $\iota_Q(x) = 0$.
\end{itemize}

The rest of the proof follows exactly as in the proof of (1) above.
\end{proof}

\begin{proof}[Proof of Corollary~\ref{FDTC > 1 tight}]
If $K \subset M$ is the fibred knot giving the open book decomposition $(\Sigma, \phi)$, then the open book decomposition associated to $\overline{K}$ is $(\Sigma,\phi^{-1})$, and if $FDTC(\phi) > 1$, then $FDTC(\phi^{-1}) < -1$.  By \cite{BEVHM}, the contact structure $\xi_{\overline{K}_{-1}}$ is supported by $(\Sigma, \phi^{-1} \circ \tau_{\bd})$, where $\tau_{\bd}$ is a positive Dehn twist around the binding of $\Sigma$.  By \cite{KR}, $\rm{FDTC}(\phi^{-1} \circ \tau_{\bd}) = \rm{FDTC}(\phi^{-1}) + 1 < 0$, and thus $\xi_{\overline{K}_{-1}}$ is overtwisted, by \cite{HKM:right1}.  Since $\cplus(\xi_{\overline{K}_{-1}}) = 0$, the corollary now follows from Theorem~\ref{rationalmaintheorem}.
\end{proof}

\begin{proof}[Proof of Corollary~\ref{L-space knot}]
If $M_r(K)$ is an L-space for some $r > 0$, then $\left(-M\right)_{-r}(\overline{K})$ is also an L-space for the same $r$.  This means that $\cplusred(\xi_{\overline{K}_{-r}})$ must vanish, and so we conclude that $\cplus(\xi_K) \neq 0$, by Theorem~\ref{rationalmaintheorem}.
\end{proof}

\begin{proof}[Proof of Corollary~\ref{non-planar contact structures}]
Since $\cplus(\xi_{\overline{K}}) = 0$, Theorem~\ref{rationalmaintheorem} implies that $\cplusred(\xi_{K_r}) \neq 0$.  Thus, $(M_r(K), \xi_{K_r})$ cannot be planar, by \cite{OSS:planar}.
\end{proof}

\begin{proof}[Proof of Corollary~\ref{non-planar Legendrian knots}]
If $L$ is a Legendrian approximation of $K$, then some negative contact surgery on $L$ is equivalent to some admissible transverse surgery on $K$, by \cite{BE:transverse}.  We know from \cite[Theorem~5.10]{Onaran} that if $L$ sits on the page of a planar open book for $\xi_K$, then the result of negative contact surgery on $L$ is also planar, which contradicts Corollary~\ref{non-planar contact structures}.  Since after putting $-L$ on the page of an open book, we can reverse the orientation to get $L$, we see that $-L$ is also not planar.
\end{proof}

\begin{proof}[Proof of Corollary~\ref{cplusred preserved}]
By Theorem~\ref{rationalmaintheorem}, $\cplus(\xi_{\overline{K}_{-r}}) = 0$ for sufficiently large $r$, say for $r > N$.  Then, for any $s > N$, we claim that $\cplus(\xi_{K_s}) \neq 0$.  Indeed, if it were not the case, then by Theorem~\ref{rationalmaintheorem}, sufficiently negative admissible transverse surgery on $\overline{K}_{-s}$ would have non-vanishing $\cplus$ invariant.  However, sufficiently negative admissible transverse surgery on $\overline{K}_{-s}$ gives us $\xi_{\overline{K}_{-s'}}$, for some $N < s' < s$, which we know to have vanishing $\cplus$ invariant.
\end{proof}

Let $K \subset (M, \xi_K)$ be an integrally null-homologous fibred knot, where $\cplus(\xi_K) \neq 0$.  Let
$$R^+_{\rm{red}}(K) = \inf \big\{r \in \Q\,\big|\,\cplusred(\xi_{K_r}) \neq 0\big\},$$
and let
$$R^+(K) = \inf \big\{r \in \Q\,\big|\,\cplus(\xi_{K_r}) \neq 0\big\}.$$
If $\cplusred(\xi_K) = 0$, then we define $R^+_{\rm{red}}(K) = \infty$.  Then, the same approach as the proof of Corollary~\ref{cplusred preserved} will give the following.

\begin{cor}\label{Rplus(red)}
Let $K \subset M$ be an integrally null-homologous fibred knot, such that $\cplus(\xi_{\overline{K}}) \neq 0$.  Then:
\begin{itemize}
\item $\cplus(\xi_{K_r}) = 0$ for all $-\infty \leq r \leq -R^+_{\rm{red}}(\overline{K})$.
\item $\cplus(\xi_{K_r}) \neq 0$ and $\cplusred(\xi_{K_r}) = 0$ for all $-R^+_{\rm{red}}(\overline{K}) < r \leq -R^+(\overline{K})$.
\item $\cplusred(\xi_{K-r}) \neq 0$ for all $-R^+(\overline{K}) < r < 0$.
\end{itemize}
\end{cor}

\section{Positive Surgeries}
\label{sec:pos surgeries}

We have thus far restricted our attention to negative surgeries, as in general, the existence of admissible transverse $r$-surgeries on a fibred transverse knot $K \subset (M, \xi_K)$ are only guaranteed for $r$ less than the page slope.  In fact, if $U$ is the unknot in $(S^3,\xi_U)$, then we know that we cannot do admissible transverse $r$-surgery for any $r \geq 0$.  In general, however, if we can thicken the standard neighbourhood of $K$, we can define admissible transverse $r$-surgery for some $r \geq 0$.

Recall from Remark~\ref{multiple tori} that the result of admissible transverse surgery may depend on the chosen neighourhood, and so is not in general well-defined.  Thus far, we have chosen a neighourhood of $K$ that was compatible with the open book decomposition, such that Theorem~\ref{thm:OBD surgery} applied.  For positive surgeries, there is no such choice.  In particular, if $K_r$ is the surgery dual knot to $K$ in $M_r(K)$, then the contact structure coming from admissible transverse $r$-surgery on $K \subset (M, \xi_K)$ bears no relation to $\xi_{K_r}$ (see Remark~\ref{admiss vs inadmiss}).

For the rest of this section, let $K \subset M$ be a null-homologous transverse knot.  If $K$ has a Legendrian approximation $L_n$ with $tb(L_n) = n$, then for any $r < n$, we can find a standard neighbourhood of $K$ inside a standard neighbourhood of $L_n$ which allows admissible transverse $r$-surgery for any $r < n$.

We will focus on the case where $\xi_K$ is overtwisted, as we can say the most in that situation.  Etnyre and van Horn-Morris proved in \cite{EVHM:fibered} that the complement of the fibred transverse knot $K \subset (M,\xi_K)$ is non-loose, {\it ie.\ }the restriction of $\xi_K$ to $M\backslash K$ is tight.  Thus $K$ intersects every overtwisted disc in $(M,\xi_K)$.  Following \cite{BO:non-loose}, we define the {\it depth} $d(K)$ of a transverse knot $K$ to be the minimum of $\big| K \cap D\big|$ over all overtwisted discs $D$.  Since $K$ is non-loose, $d(K) \geq 1$.  The depth of a transverse link, a Legendrian knot, and a Legendrian link can be similarly defined.

We first show that we can pass from transverse knots to Legendrian knots without increasing the depth.

\begin{lem}\label{lem:depth transverse to legendrian}
If $K$ is a transverse knot such that $d(K) = 1$, then there exists a Legendrian approximation $L$ of $K$ such that $d(L) = 1$.
\end{lem}
\begin{proof}
We consider a neighbourhood of $\{r = \pi/4, \theta = 0\} \subset (S^1 \times \R^2,\xi_{\rm{rot}} = \ker(\cos r\, dz + r\sin r\,d\theta))$ as a generic model for a neighbourhood of $K$ (where $z \sim z + 1$), and let the overtwisted disc $D = \{z = 0, r \leq \pi\}$.  It would seem preferable to consider $\{r = 0\}$, but since that knot passes through the unique elliptic point of the characteristic foliation on $D$, it would not represent a generic model.

The contact planes away from $r = 0$ are spanned by $\frac{\bd}{\bd r}$ and the vector $X(r) = -\tan r\,\frac{\bd}{\bd z} + \frac{\bd}{\bd \theta}$.  We define a Legendrian approximation $L_{-n}$ of $K'$ by
$$L_n(t) = \begin{pmatrix}\alpha\, z(t)\\\pi/4 + \epsilon \cos(nt)\\ -\alpha f(nt)\end{pmatrix},$$
in the coordinates $(z,r,\theta)$, for some positive real number $\alpha$, where $$z'(t) = -\frac{n f'(nt)}{\tan\left(\pi/4 + \epsilon \cos(nt)\right)}, \, z(0) = 0.$$  We choose $f : \R \to \R$ to be a smooth function such that
$$f(t) = \begin{cases}\frac{\sin t}{m} & t \in \left(0, \frac{\pi}{2}-\delta\right)\cup\left(\frac{3\pi}{2} + \delta, 2\pi\right), \\ m\sin t & t \in \left(\frac{\pi}{2} + \delta, \frac{3\pi}{2} - \delta\right),\end{cases}$$
where $m$ is a positive integer and $\delta > 0$ is small. We define $f$ to be periodic with period $2\pi$, and we define $f$ on $\left(\frac{\pi}{2} - \delta, \frac{\pi}{2} + \delta\right)\cup\left(\frac{3\pi}{2} - \delta, \frac{3\pi}{2} + \delta\right)$ such that $f$ is smooth.

Then, given any $\epsilon > 0$, choose $n$ to be a sufficiently large positive integer such that $z(2\pi) > 1$ and $\max |f|/z(2\pi) < \arctan \left(4\epsilon/\pi\right)$; choose $m$ to be a sufficiently large positive integer such that there exists some value $z_0$ of $z$ such that $\big| L_n \cap \{z = z_0 \}\big| = 1$.  If we now let $\alpha$ be $1/z(2\pi)$, then $L_{-n}$ will be a closed Legendrian in an $\epsilon$-neighbourhood of $K'$.  Note that $tb(L_{-n}) = -n$, and $\big|L_{-n}\cap\{ z = z_0/z(2\pi), r \leq \pi\}\big| = 1$.  Thus, $d(L_{-n}) \leq 1$.

To show that $d(L_{-n}) \neq 0$, note that a transverse push-off of $L_{-n}$ can be made arbitrarily close to $L_{-n}$.  Thus, if there exists an overtwisted disc in $(M, \xi)$ that is disjoint from $L_{-n}$, it is also disjoint from a sufficiently small neighbourhood of $L_{-n}$.  Then, the tranverse push-off $K$ of $L_{-n}$ would also be disjoint from this overtwisted disc.  However, $d(K) = 1$, and so this cannot happen.
\end{proof}

The next step is to show that if $d(L) = 1$, then $L$ can be negatively destabilised.

\begin{lem}\label{lem:destabilise along OT disc}
Let $K \subset (M,\xi)$ be a non-loose transverse knot, and let $L_n$ be a Legendrian approximation of $K$ with $tb(L_n) = n$.  If $d(K) = d(L_n) = 1$, then $L_n$ can be negatively destabilised to $L_m$ (hence $L_m$ is also a Legendrian approximation of $K$) for any $n \geq m$, where $tb(L_m) = m$ and $d(L_m) = 1$.
\end{lem}
\begin{proof}
Let $T$ be the convex boundary of a regular neighbourhood of $L_n$.  Let $\gamma$ be a curve on $T$ isotopic to a meridian of $L_n$.  Using the Legendrian Realisation Principle (\cite[Theorem~3.7]{Honda:classification1}), we can isotope $T$ to another convex torus $T'$ such that $\gamma$ is Legendrian with $tb(\gamma) = -1$.

Given an overtwisted disc $D$ such that $\big|L_n \cap D\big| = 1$, isotope $D$ such that its intersection with $T'$ is exactly $\gamma$.  Let $A \subset D$ be the annulus with boundary $\bd D \cup \gamma$.  Since the twisting of the contact planes along the boundary components of $A$ is non-positive with respect to the framing given by $A$, we can perturb $A $ to be convex. Consider the dividing curves $\Gamma$ on $A$: since $tb(\gamma) = -1$ and $tb(\bd D) = 0$, $\Gamma$ intersects $\gamma$ twice and $\bd D$ zero times.  Since $L_n$ is a non-loose Legendrian knot (as $K$ is a non-loose transverse knot), the dividing curves $\Gamma$ can have no contractible components, by Giroux's Criterion (\cite{Giroux:convex}), so, $\Gamma$ is an arc parallel to $\gamma$, and gives a bypass.

If $T''$ is the convex torus we get from the bypass $A$ on $T'$, the dividing curves on $T''$ are meridional.  The region in between $T'$ and $T''$ is a basic slice, and we can find convex tori with dividing curves of slope $m$ for any $n \leq m$.  We define $L_m$ to be a Legendrian divide of a convex torus with dividing curves of slope $m$.

In order to determine whether $L_m$ is a positive or negative destabilisation of $L_n$, consider the convex torus $T'''$, the boundary of a regular neighbourhood of $L_{n-1}$, the negative stabilisation of $L_n$ (we pick a negative stabilisation since it is non-loose).  If we glue the basic slice from $T'''$ to $T'$ to the basic slice from $T'$ to $T'''$, we get a $T^2 \times I$ with a tight contact structure, and convex boundaries with dividing curves of slope $n-1$ and $\infty$, {\it ie.\ }another basic slice.  By the classification in \cite{Honda:classification1} of tight contact structures on basic slices, the signs of the basic slices that we glue together must agree.  Thus, since $L_{n-1}$ is a negative stabilisation of $L_n$, so too is $L_n$ an $(m-n)$-fold negative stabilisation of $L_m$.

To show that $d(L_m) \geq 1$, note (as in the proof of Lemma~\ref{lem:depth transverse to legendrian}) that $K$ is a transverse push-off of $L_m$, so $L_m$ must be non-loose.  Then, since $\big|L_m \cap D\big| = 1$, $d(L_m) \leq 1$, and so we conclude that $d(L_m) = 1$.
\end{proof}

Finally, we use these destabilisations to define admissible transverse $r$-surgery on $K$, for $r \geq 0$.

\begin{proof}[Proof of Theorem~\ref{positive surgeries}]
Ito and Kawamuro showed in \cite[Corollary~3.4]{IK:coverings} that the binding of an open book $(\Sigma, \phi)$ is a transverse link of depth 1 if and only if the monodromy is not right-veering.  With our hypotheses, we conclude that $d(K) = 1$.

Let $(\Sigma, \phi)$ be the open book defined by $K$.  By Theorem~\ref{thm:OBD surgery}, an open book for admissible transverse $(-1/n)$-surgery on $K$ is given by $(\Sigma,\phi\circ\tau_{\bd}^n)$, where $\tau_{\bd}$ is a Dehn twist about the binding component of $\Sigma$.  By Corollary~\ref{FDTC > 1 tight}, this supports a contact structure $\xi_{K_{-1/n}}$ that satisfies $\cplusred(\xi_{K_{-1/n}}) \neq 0$, for sufficiently large integers $n$.

Since $d(K) = 1$, Lemma~\ref{lem:destabilise along OT disc} guarantees the existence of Legendrian approximations $L_m$ of $K$ with $tb(L_m) = m$, for any $m$. We will use the neighbourhood of $K$ coming from a standard neighbourhood of $L_m$ to define transverse admissible $r$-surgery on $K$, for $r < m$. This neighbourhood also corresponds to a neighbourhood of $K_{-1/n} \subset (M_{-1/n}(K), \xi_{K_{-1/n}})$, and admissible transverse $r$-surgery on $K \subset (M, \xi_K)$ for $-1/n < r < m$ is equivalent to some admissible transverse surgery on $K_{-1/n} \subset (M_{-1/n}(K), \xi_{K_{-1/n}})$ using the corresponding neighbourhood. The theorem then follows from Lemma~\ref{lemma:admissible preserves non-vanishing}.
\end{proof}

\begin{proof}[Proof of Corollary~\ref{where surgery is not tight}]
The first statement, for $\HFred(M) = 0$, follows from Theorems~\ref{maintheorem}~and~\ref{positive surgeries}.

For the second statement in the theorem, consider now a general $M$.  The manifold $M_0(K)$ supports a tight contact structure in both its orientations (see for example \cite{HKM:haken}).  If $\cplusred(\xi_K)$ and $\cplusred(\xi_{\overline{K}})$ both vanish, then $\cplus(\xi_{K_r})$ and $\cplus(\xi_{\overline{K}_r})$ do not vanish for any $r < 0$, by Theorem~\ref{rationalmaintheorem}.  So assume $\cplusred(\xi_K) \neq 0$.  By Lemma~\ref{lemma:admissible preserves non-vanishing}, $\cplus(\xi_{K_r}) \neq 0$ for all $r < 0$.  By Corollary~\ref{Rplus(red)}, at least one of $\cplus(\xi_{K_r})$ and $\cplus(\xi_{\overline{K}_{-r}})$ is non-vanishing for any $r > 0$, unless $r = R^+(K) = R^+_{\rm red}(K)$.

For the last statement, note that at least one of $\phi_K$ and $\phi_{\overline{K}}$ is not right-veering (unless $M = \#\left(S^1 \times S^2\right)$ and $\phi_K = \rm{id}$, whereupon $\HFred(M) = 0$, and we're done, as above).  Assume $\phi_{\overline{K}}$ is not right-veering, so $\xi_{\overline{K}}$ is overtwisted. Then $(M_r(K), \xi_{K_r})$ is tight for all $r < 0$, by Theorem~\ref{rationalmaintheorem}, and $-\left(M_r(K)\right) = (-M)_{-r}(\overline{K})$ supports a tight contact structure for all $r < 0$, by Theorem~\ref{positive surgeries}.
\end{proof}

\begin{figure}[htbp]
\begin{center}
\begin{overpic}[scale=1.3,tics=20]{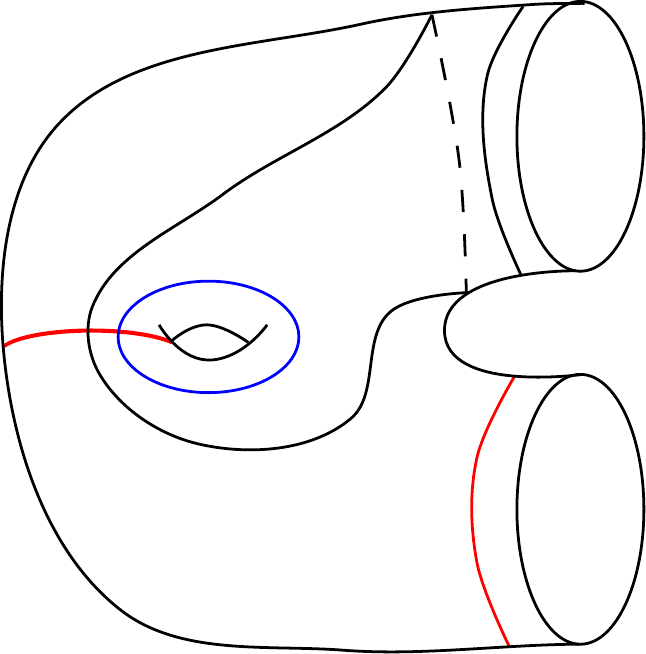}
\put(100,140){\large $\alpha$}
\put(18,125){\large $\beta$}
\put(185,210){\large $\delta_1$}
\put(187,80){\large $\delta_2$}
\put(120,185){\large $\gamma$}
\end{overpic}
\caption{The open book $(\Sigma,\tau_\alpha^{-1}\tau_\beta\tau_{\delta_2})$ is a positive stabilisation of the open book given by the figure-eight knot $K\subset S^3$. The curve $\delta_1$ is a push-off of $K$, and $\gamma$ is the boundary of an overtwisted disc in $\xi_K$.}
\label{fig:figure eight OBD}
\end{center}
\end{figure}
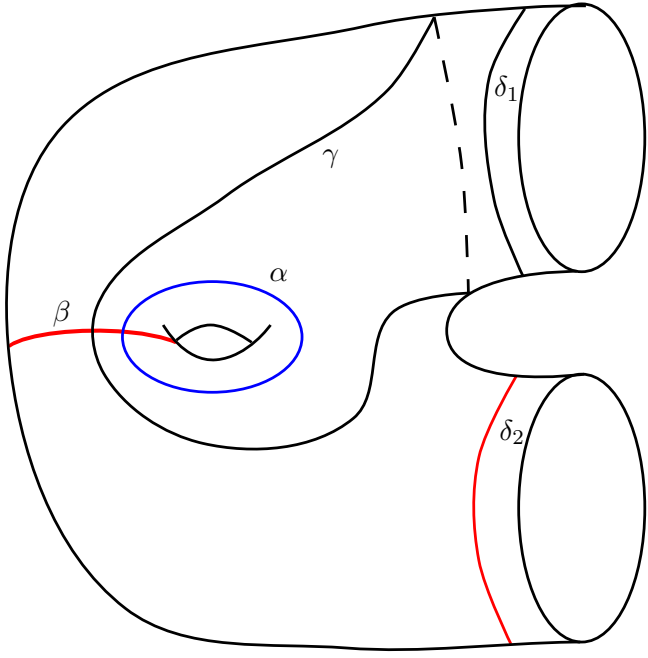

If the open book $(\Sigma, \phi)$ is a negative stabilisation of another open book (whereupon it is not right-veering), Baker and Onaran explicitly identified an overtwisted disc that intersects $K$ once and has boundary on the page of a stabilisation of $(\Sigma, \phi)$ (see \cite[Theorem~5.2.3]{BO:non-loose}).  Given a transverse knot $K$, a Legendrian approximation $L$ of $K$, and an overtwisted disc $D$ such that $\big| L \cap D \big| = 1$, a careful tracking of dividing curves will show that contact $(-1)$-surgery on $\bd D$ results in a contact manifold contactomorphic to the starting one, but where $L$ gets identified with its negative destabilisation.  We can use this and Baker and Onaran's identification of $\bd D$ on a page to construct open books for admissible transverse $r$-surgery on $K$ for $r \geq 0$.

\begin{example}
Let $K \subset (S^3,\xi_K)$ be the figure-eight knot.  After a positive stabilisation using a boundary-parallel arc, we use \cite[Theorem~5.2.3]{BO:non-loose} to identify $\gamma$ in Figure~\ref{fig:figure eight OBD} as the boundary of an overtwisted disc $D$ which $K$ intersects once.

In the construction of $S^3$ from the open book decomposition $(\Sigma,\tau_\alpha^{-1}\tau_\beta\tau_{\delta_2})$, the boundary component parallel to $\delta_1$ is identified with the figure-eight knot $K$, and the framing on $K$ given by $\Sigma$ is $-1$ (with respect to the Seifert framing).  The open book  $(\Sigma,\tau_\alpha^{-1}\tau_\beta\tau_\gamma^n\tau_{\delta_2})$, for some $n \in \Z$, also supports $(S^3,\xi_K)$, but now the framing given to $K$ by $\Sigma$ is $n-1$, since composing the monodromy with $\tau_{\gamma}^n$ is equivalent to contact $(-1)$-surgery on $n$ push-offs of $\gamma = \bd D$.  Thus, the open book $(\Sigma,\tau_\alpha^{-1}\tau_\beta\tau_\gamma^n\tau_{\delta_1}\tau_{\delta_2})$ supports the tight contact manifold $(S^3_{n-2}(K),\xi_{n-2})$, where $\xi_{n-2}$ is the result of admissible transverse $(n-2)$-surgery on $K$.
\end{example}

\section{Further Questions}
\label{sec:questions}

In this section, we discuss how Theorem~\ref{rationalmaintheorem} might be brought to bear on some open questions.

\subsection{$\cplus$ versus $\chat$}
\label{subsec:cplus versus chat}

Given a contact manifold $(M, \xi)$, we know that $\cplus(\xi)$ vanishes if $\chat(\xi)$ does.  The converse, however, remains open: does the vanishing of $\cplus(\xi)$ imply the vanishing of $\chat(\xi)$?

To explore this, we define two invariants of elements of $\HFhat(M)$, using $\CF^{0,i}(M) = \ker U^i \subset \CFplus(M)$, and its homology $\HF^{0,i}(M)$.

\begin{definition}
Given $x \in \HFhat(M)$, let the \textit{vanishing order} $\VO(x)$ of $x$ be the minimum $i$ such that the image of $x$ in $\HFhat(M) \to \HF^{0,i}(M)$ is trivial (and it equals $\infty$ is no such finite $i$ exists).  Given $x \neq 0 \in \HFhat(M)$, let the \textit{dual vanishing order} $\DVO(x)$ of $x$ be the minimum of $\VO(y)$ over all $y \in \HFhat(-M)$ that pair non-trivially with $x$ (and it equals $\infty$ if no such $y$ exists).
\end{definition}

Note that $\chat(\xi) \neq 0$ if and only if $\VO(\chat(\xi)) > 0$; $\cplus(\xi) \neq 0$ if and only if $\VO(\chat(\xi)) = \infty$; and $\cplusred(\xi) \neq 0$ if and only if $\VO(\chat(\xi)) = \infty$ and $\DVO(\chat(\xi)) < \infty$.

Let $K \subset M$ be a fibred knot, fix an integer $m \geq 1$, and consider the following short exact sequences, where the bottom-right complex $C = \mathcal{C}_{\overline{K}}\{1 \leq \max(i,j-g+1) \leq m\}\cup\mathcal{C}_{\overline{K}}\{i = m, j = g+m\}$.

\begin{center}
\begin{tikzpicture}[node distance=3cm,auto]
  \node (UL) {$\mathcal{C}_{\overline{K}}\{i = 0, j \leq g-1\}$};
  \node (ULL) [left of=UL, node distance = 3cm] {$0$};
  \node (UM) [right of=UL, node distance=4cm] {$\mathcal{C}_{\overline{K}}\{i = 0\}$};
  \node (UR) [right of=UM, node distance=4cm] {$\mathcal{C}_{\overline{K}}\{i = 0, j = g\}$};
  \node (URR) [right of=UR, node distance = 3cm] {$0$};
  \node (LL) [below of=UL, node distance = 2cm] {$\mathcal{C}_{\overline{K}}\{i = 0, j \leq g-1\}$};
  \node (LLL) [left of=LL, node distance = 3cm] {$0$};
  \node (LM) [right of=LL, node distance=4cm] {$\mathcal{C}_{\overline{K}}\{0 \leq i \leq m\}$};
  \node (LR) [right of=LM, node distance=4cm] {$C$};
  \node (LRR) [right of=LR, node distance = 3cm] {$0$};
  \draw[->] (ULL) to (UL);
  \draw[->] (UL) to (UM);
  \draw[->] (UM) to (UR);
  \draw[->] (UR) to (URR);
  \draw[->] (LLL) to (LL);
  \draw[->] (LL) to (LM);
  \draw[->] (LM) to (LR);
  \draw[->] (LR) to (LRR);
  \draw[->] (UL) to (LL);
  \draw[->] (UM) to (LM);
  \draw[->] (UR) to (LR);
\end{tikzpicture}
\end{center}

Then, the same proof as in Theorem~\ref{rationalmaintheorem}(1), and the similar version corresponding to Theorem~\ref{rationalmaintheorem}(2), gives the following.

\begin{thm}\label{vanishing order}
Let $K \subset M$ be a fibred knot.  Then for all sufficiently negative integers $n$:
\begin{enumerate}
\item $\DVO(\chat(\xi_{\overline{K}})) - 1 \leq \VO(\chat(\xi_{K_n})) \leq \DVO(\chat(\xi_{\overline{K}}))$.
\item $\DVO(\chat(\xi_{K_n})) - 1 \leq \VO(\chat(\xi_{\overline{K}})) \leq \DVO(\chat(\xi_{K_n}))$.
\end{enumerate}
\end{thm}

\begin{remark}
The inequalities are there because $C$ is not precisely $\ker \left(U|_{A^+}\right)^i$ for any $i$.  In fact, with the identification of $A^+$ with $\mathcal{C}_{\overline{K}}\{\max(i,j-g+1)\geq 1\}$, we have $\ker \left(U|_{A^+}\right)^m \subsetneq C \subsetneq \ker \left(U|_{A^+}\right)^{m+1}$.
\end{remark}

If we could find a fibred knot $K$ supporting $(M, \xi_K)$, such that $\cplusred(\xi_K) \neq 0$ and $\DVO(\xi_K) > 1$, then Theorem~\ref{vanishing order} would imply that for all sufficiently negative integers $n$, $\cplus(\xi_{\overline{K}_n}) = 0$, but $\chat(\xi_{\overline{K}_n}) \neq 0$.

\begin{question}
Do there exist examples of contact manifolds $(M, \xi)$ with $\DVO(\chat(\xi)) > 1$? Is there any geometric meaning to $\DVO(\chat(\xi))$?
\end{question}

\subsection{Genus-1 Open Books}
\label{subsec:genus1}

As yet no obstruction has been found for a non-planar contact manifold $(M, \xi)$ to be supported by a genus-$1$ open book decomposition.  In fact, it is currently unknown whether there exist any contact manifolds whose {\it support genus} ({\it ie.\ }minimal genus of a supporting open book decomposition, defined in \cite{EO:invariants}) is greater than 1.  To explore this, we look at properties of genus-1 open books with one boundary component.

Let $K\subset M$ be a fibred knot of genus 1.  A result of Honda, Kazez, and Matić (\cite{HKM:contactclass}), and independently Baldwin (\cite{Baldwin:genus1}), shows that $\xi_K$ is tight if and only if $\cplus(\xi_K)$ is non-vanishing.  Honda, Kazez, and Matić prove that this is equivalent to the monodromy $\phi_K$ of the open book having {\it fractional Dehn twist coefficient} $\rm{FDTC}(\phi_K)$ greater than 0, or greater-than-or-equal to 0 if $\phi_K$ is a periodic mapping class (see \cite{HKM:contactclass} for details on these terms).

\begin{thm}\label{thm:genus 1}
Let $K \subset M$ be a fibred knot of genus 1, and let $(M_{r}(K),\xi_{K_r})$ be the result of admissible transverse $r$-surgery on $K \subset (M,\xi_K)$, for $r < 0$. The following are equivalent:
\begin{enumerate}
\item $\cplusred(\xi_{\overline{K}}) \neq 0$ and $\DVO(\chat(\xi_{\overline{K}})) = 1$.
\item $\cplus(\xi_{K_r}) = 0$ for some $r < 0$.
\item $\chat(\xi_{K_r}) = 0$ for some $r < 0$.
\item $\cplus(\xi_{K_{-1}}) = 0$.
\item $\chat(\xi_{K_{-1}}) = 0$.
\item $\xi_{K_{-1}}$ is overtwisted.
\item \begin{enumerate}\item $\rm{FDTC}(\phi_{\overline{K}}) > 1$, if $\phi_{\overline{K}}$ is periodic. \item $\rm{FDTC}(\phi_{\overline{K}}) \geq 1$, if $\phi_{\overline{K}}$ is not periodic. \end{enumerate}
\end{enumerate}
\end{thm}
\begin{proof}
$(1) \implies (2)$: This is Theorem~\ref{rationalmaintheorem}.

$(4) \iff (5) \iff (6)$: These follow from the discussion preceding the theorem.

$(5) \implies (3) \implies (2)$: These are trivial.

$(2) \implies (4)$: We show the contrapositive. We claim that inadmissible transverse $\frac{n}{n-1}$-surgery on $K_{-1}$, for $n > 1$ an integer, results in the same contact manifold as admissible transverse $(-n)$-surgery on $K$, namely $(M_{-n}(K),\xi_{K_{-n}})$ ({\it cf.\ }\cite[Section~3.3]{Conway}).  Indeed, let $\lambda$ be the $0$-framing on a neighbourhood of $K$, and let $\mu$ be a meridional curve for $K$.  Then the page of the open book for $K_{-1}$ also traces the curve $\lambda$ on the boundary of a neighbourhood of $K_{-1}$, but now the meridional curve is $\mu - \lambda$.  Inadmissible transverse $\frac{n}{n-1}$-surgery on $K_{-1}$, measured in terms of $\lambda$ and $\mu$, is $$n\cdot (\mu - \lambda) + (n-1)\cdot \lambda = n\cdot \mu - \lambda,$$ which is $(-n)$-surgery on $K$.  That the result of inadmissible transverse surgery $K_{-1}$ can be considered as a single transverse surgery on $K$ follows from the definition of transverse surgery.

By \cite[Theorem~1.6]{Conway}, if $\cplus(\xi_{K_{-1}}) \neq 0$, then inadmissible transverse $r$-surgery on $K_{-1}$ preserves this non-vanishing for all $r > 1$. In particular, $\cplus(\xi_{K_{-n}}) \neq 0$ for all integers $n > 0$.  By Lemma~\ref{lemma:admissible preserves non-vanishing}, this implies that $\cplus(\xi_{K_r}) \neq 0$ for all $r < 0$.

$(3) \implies (1)$: The only part of this that does not follow from Theorem~\ref{rationalmaintheorem} is the dual vanishing order of $\chat(\xi_{\overline{K}})$.  However, if $\chat(\xi_{K_r}) = 0$ some $r < 0$, then Lemma~\ref{lemma:admissible preserves non-vanishing} implies that $\chat(\xi_{K_n}) = 0$ for all sufficiently negative integers $n$, \textit{ie.\@} $\VO(\chat(\xi_{K_n})) = 0$.  Thus, by Theorem~\ref{vanishing order}, $\DVO(\chat(\xi_{\overline{K}})) \leq 1$.  However, since $\cplusred(\xi_{\overline{K}}) \neq 0$, we know that $\DVO(\chat(\xi_{\overline{K}})) \geq 1$, and so therefore it is equal to $1$.

$(4) \iff (7)$: Note that an open book decomposition for $(M_{-1}(K),\xi_{K_{-1}})$ is given by $(\Sigma_{1,1},\phi_K\circ\tau_{\bd})$, where $\Sigma_{1,1}$ is genus 1 surface with one boundary component, and $\tau_{\bd}$ is a positive Dehn twist about a curve parallel to $\bd\Sigma_{1,1}$.  Furthermore, $\rm{FDTC}(\phi_K \circ \tau_{\bd}) = \rm{FDTC}(\phi_K) + 1$ and $\rm{FDTC}(\phi_K) = -\rm{FDTC}(\phi_{\overline{K}})$, according to \cite{KR}.  Thus, the equivalence follows from the discussion preceding this theorem.
\end{proof}

Consider now a contact manifold $(M, \xi)$, supported by a genus-$1$ open book $(\Sigma, \phi)$, where $\cplusred(\xi) \neq 0$. We can cap off $(\Sigma, \phi)$ to an open book $(\Sigma_{1,1},\phi')$ with a single binding component that supports $(M',\xi')$, where $\cplusred(\xi') \neq 0$, $\DVO(\chat(\xi')) = 1$, and $\rm{FDTC}(\phi') > 1$, by Theorem~\ref{thm:genus 1}.

\begin{question}\label{genus1question}
By \cite[Theorem~1.2]{Baldwin:cappingoff}, we know that $\DVO(\chat(\xi)) \geq \DVO(\chat(\xi'))$.  Under what conditions can we say that $\DVO(\chat(\xi')) = 1$ implies $\DVO(\chat(\xi)) = 1$?  Under what conditions does $\rm{FDTC}(\phi') > 1$ imply that $\rm{FDTC}(\phi, B) > 1$ at every boundary component $B$ of $\Sigma$?
\end{question}

Even partial answers to Question~\ref{genus1question} could lead to obstructions to $\xi$ being supported by a genus-1 open book.

\bibliography{references}{}
\bibliographystyle{plain}
\end{document}